\documentclass[12pt]{article}
\usepackage{amsfonts,amssymb,amsmath,amsthm,bm,mathtools}
\usepackage{latexsym,url,color,enumerate}
\usepackage{array,amstext}

\newtheorem{theorem}{Theorem}[section]

\newtheorem{lemma}[theorem]{Lemma}

\newtheorem{assumption}[theorem]{Assumption}

\theoremstyle{definition}

\theoremstyle{remark}
\newtheorem{remark}[theorem]{Remark}
\newcommand{\RNum}[1]{\uppercase\expandafter{\romannumeral #1\relax}}

\newcommand{\ZZ}{\mathbb{Z}}%Integers
\newcommand{\NN}{\mathbb{N}}%Natural numbers
\newcommand{\QQ}{\mathbb{Q}}%Rationals
\newcommand{\FF}{\mathbb{F}}%Finite field
\newcommand{\diag}{\textup{diag}}

\newcommand{\Hom}{\textup{Hom}}
\newcommand{\End}{\textup{End}}
\newcommand{\rk}{\textup{rk}}

\newcommand{\GL}{\textup{GL}}

\newcommand{\SL}{\textup{SL}}

\newcommand\nc\newcommand

\title{On conjugacy of diagonalizable integral matrices}
\author{Gabriele Nebe
\thanks{Lehrstuhl D f\"ur Mathematik, RWTH Aachen University, nebe@math.rwth-aachen.de} }

\begin{document}

\maketitle

{\small 
{\sc Abstract.} 
It is shown that under some additional assumption 
two diagonalizable integral matrices $X$ and $Y$ 
with only rational eigenvalues are conjugate in $\GL_n(\ZZ )$ if
and only if they are conjugate over all localizations. 
This is used to prove that for a prime $p \equiv 3\pmod{4}$ 
the adjacency matrices of the 
Paley graph and the Peisert graph on $p^2$ vertices are conjugate in
$\GL_{p^2}(\ZZ )$, answering a question by Peter Sin \cite{Sin}.
}

\section{Introduction}

Let $X$, $Y$ $\in \ZZ^{n\times n} $  be two integral matrices. 
Then $C(X) := \{ A \in \ZZ ^{n\times n } \mid AX = XA \} $ is a
$\ZZ $-order and $C(X,Y) := \{ A\in \ZZ^{n\times n} \mid 
AX = Y A \} $ is a right module for $C(X)$. 
Faddeev  \cite{Fadeev} shows that $X$ and $Y$ are conjugate in 
$\GL_n(\ZZ )$ if and only if $C(X,Y)$ is a free $C(X)$-module. 

Local-global properties for similarity of matrices have been considered 
for lattices over orders in \cite{Jacobinski} and later   
in \cite{Guralnick}.
Using the above mentioned result by Faddeev both papers, 
\cite[Satz 7]{Jacobinski} 
and \cite[Theorem 7]{Guralnick}, show that 
two matrices over the
ring of integers in an algebraic number field  are conjugate over all 
localizations if and only if they are conjugate over the ring of
integers in some finite field extension. 
In certain cases, there is no need to pass to an extension field. 
This paper gives an 
additional sufficient condition (see Assumption \ref{ass}) for which a thorough analysis 
of  \cite{Jacobinski} allows to prove 
 Theorem \ref{main} saying that,
 two diagonalizable integral matrices satisfying Assumption \ref{ass} 
 %and having only rational eigenvalues
 are conjugate in $\GL_n(\ZZ )$ if
and only if they are conjugate over all localizations.

The work on this paper started during the  Hausdorff Trimester program 
``Logic and Algorithmic group theory''. 
I thank the HIM for their support during this program and 
Eamonn O'Brien for communicating 
a question by Peter Sin which was the main motivation for this  note.
The recent paper \cite{Sin} shows that for any prime $p\equiv 3\pmod{4}$ 
the adjacency matrices of the Paley graph $A(p^2)$ and the Peisert graph  $A^*(p^2)$
on $p^2$ vertices
are conjugate over all localizations of $\ZZ $ and asks whether 
these are also conjugate in $\GL_{p^2}(\ZZ )$. 
As these adjacency matrices are rationally diagonalizable 
and satisfy Assumption \ref{ass} (see Section \ref{Paley}) 
Theorem \ref{main} implies a positive answer to this question. 

The paper contributes
to the  SFB TRR 195 ``Symbolic Tools in Mathematics and their Application''.

\section{Notation and statement of main result}\label{notation}

We denote by $\ZZ $ the ring of integers in the rationals $\QQ $. 
For a prime $p$  let
$$\ZZ_{(p)} := \{ \frac{a}{b} \in \QQ \mid  p \mbox{ does not divide } b \} $$
 denote the localization of $\ZZ $ at $p$. 
For $n\in \NN $ let 
$$\GL_n(\ZZ ):= \{ g\in \ZZ ^{n\times n} \mid \det(g) \in \{ \pm 1 \} \} $$ 
be 
the group of invertible integral matrices of size $n$ and 
$$\GL_n(\ZZ _{(p)} ):= \{ g\in \ZZ _{(p)} ^{n\times n} \mid p \mbox{ does not divide } \det(g)  \} $$ 
the group of invertible matrices  over $\ZZ _{(p)}$.

 Let $A\in \ZZ ^{n\times n}$. Then there are matrices $g,h\in \GL_n(\ZZ )$ such that 
 $$gAh = \diag (d_1,\ldots , d_r,0,\ldots , 0 ),\mbox{  with } d_i\in \NN , \ d_1\mid d_2 \mid \ldots \mid d_r .$$
 Then the {\em abelian invariants}  $(d_1,\ldots , d_r)$ of $A$ are uniquely determined by $A$ and 
 the {\em Smith group} of $A$ is the torsion part of the cokernel of the endomorphism $A$; as an abelian group 
 this is isomorphic to $\ZZ / d_1\ZZ \times \ldots \times \ZZ /d_r\ZZ $. Its exponent is $d_r$.

In  this note we consider integral diagonalizable matrices 
 $X,Y \in \ZZ^{n\times n}$ with the same  minimal polynomial
 $\mu_X = \mu _Y = \prod _{i=1}^k (t-a_i) \in \ZZ[t] $ where 
 $a_1,\ldots , a_k \in \ZZ $ are pairwise distinct integers. 
 Then by Chinese Remainder Theorem  the $\QQ $-algebras 
 $\QQ [X]  $ and also $\QQ [Y ]$ are isomorphic to a direct sum of copies of $\QQ $ 
 $$ \QQ [X] \cong \bigoplus _{i=1}^k \QQ [t]/(t-a_i) \cong \bigoplus _{i=1}^k \QQ .$$
 Let $e_i \in \QQ [X] \subseteq \QQ ^{n\times n}$ denote the primitive idempotents of this algebra ($1\leq i \leq k$). 
 Then there are  minimal $q_i\in \NN $ such that 
 $E_i:=q_ie_i \in \ZZ ^{n\times n} $ for all $i$. 
 For our proof of the main result we make the following assumption on the Smith group of $E_i$:

	 \begin{assumption}\label{ass}
		 Assume that one of the following two statements holds:
		 \begin{itemize}
			 \item[(a)] 
For all $1\leq i \leq k$ the Smith group of $E_i$ has exponent $q_i$.
\item[(b)] 
	$\rk (e_1) =1 $ and for all $2\leq i \leq k $ the Smith group of $E_i$ has exponent $q_i$.
	\end{itemize}
	 \end{assumption}

Though the formulation of part (b) of the assumption does not seem to be 
natural, 
 this is the situation that will occur quite frequently
in graph theory. It is the one that we need in Section \ref{Paley}.

\begin{theorem} \label{main} 
	Let $X,Y \in \ZZ^{n\times n}$ be two matrices with minimal polynomial
	$\mu_X = \mu _Y = \prod _{i=1}^k (t-a_i) \in \ZZ[t] $ where 
	 $a_1,\ldots , a_k \in \ZZ $ are pairwise distinct integers.
	 Assume that $X$ satisfies Assumption \ref{ass}.
	Then there is some $T\in \GL_n(\ZZ) $ with $TXT^{-1} = Y $
	if and only if for all primes $p$ there are matrices
	$T_p \in \GL _n(\ZZ _{(p)} ) $ with $T_pXT_p^{-1} = Y$. 
\end{theorem}

Note that we could prove Theorem \ref{main} under weaker hypotheses, 
for instance for minimal polynomials $\mu_X=\mu_Y = \prod _{i=1}^k f_i$ 
where all the pairwise distinct irreducible factors $f_i$ have 
equation orders $\ZZ [t]/(f_i(t))$ that are principal ideal domains. 
Such an assumption on the equation orders is necessary as the following example shows:
Put 
$$Y:= \begin{pmatrix} 0 & 1 \\ \mbox{-}6 & 0 \end{pmatrix} ,\ 
X:= \begin{pmatrix} 0 & 2 \\ \mbox{-}3 & 0 \end{pmatrix} ,\ 
T_2:=\begin{pmatrix} 0 & \mbox{-}1 \\ 3 & 0 \end{pmatrix} ,\ 
T_p:= \diag (1,2)  \mbox{ for } p > 2 .$$
Then
$\mu _X = \mu _Y = t^2+6 $ is irreducible but the equation order 
$\ZZ [t]/(t^2+6) \cong \ZZ [\sqrt{-6}]$ has class number 2. 
It is easy to see that $X$ and $Y$ are not conjugate in $\GL_2(\ZZ )$
but  for all primes $p$ the matrix  $T_p \in \GL_2(\ZZ _{(p)}) $ 
satisfies
$T_p X T_p^{-1} = Y$, so $X$ and $Y$ are conjugate over all localizations.

Also Assumption \ref{ass} cannot be completely omitted, 
as can be seen by taking 
$$Y:= \begin{pmatrix} 1 & 1 \\ 0 & 6 \end{pmatrix} ,\ 
X:= \begin{pmatrix} 1 & 2 \\ 0 & 6 \end{pmatrix} ,\ 
T_2:=\begin{pmatrix} \mbox{-}1 & 1 \\ 0 & 3 \end{pmatrix} ,\ 
T_p:= \diag (1,2)  \mbox{ for } p > 2 .$$
Here $\mu_X=\mu_Y = (t-1)(t-6)$ and 
$T_p X T_p^{-1} = Y$ 
for all primes $p$ but $X$ and $Y$ are not conjugate over $\GL_2(\ZZ )$.
Note that neither $X$ nor $Y$ satisfies Assumption \ref{ass} as both matrices
$$E_1= 5 e_1 = X-1 , E_2=5e_2 = 6-X $$  have trivial Smith group.

\section{Proof of Theorem \ref{main} based on \cite{Jacobinski}} \label{JProof}

For a ring $O$ we put 
$\SL_n(O ) := \{ g\in O^{n\times n} \mid \det(g) = 1 \}$. 

\begin{lemma}\label{b} (see \cite[Theorem K.14]{Jantzen})
	%for an explicit and constructive proof of a much more general statement)
        Let $q\in \ZZ $ be such that $q\geq 2$. \\
        Then the entry-wise reduction map
	        $\SL_n(\ZZ ) \to \SL _n(\ZZ/q\ZZ )$ is onto.
	\end{lemma}

It is clear that Lemma \ref{b}
cannot be true for $\GL_n$, as the determinant
of the reduction modulo $q$ of a matrix in $\GL_n(\ZZ )$ is $\pm 1$ mod $q$.

One direction of Theorem \ref{main} is obvious:
If there is a matrix $T\in \GL_n(\ZZ )$  with $TXT^{-1} =Y$ then 
we may put $T_p:=T  \in\GL_n(\ZZ _{(p)})$ for all primes $p$ to see that the
two matrices are also conjugate over all localizations. 

To see the opposite direction
we use \cite[Satz 4]{Jacobinski}. 
I thank Peter Sin for simplifying my original approach.

The ring 
$$R := \ZZ [t] /\prod _{i=1}^k (t-a_i ) $$  is a $\ZZ $-order in 
the commutative split semisimple $\QQ $-algebra 
$${\mathcal A} := \QQ [t] /\prod _{i=1}^k (t-a_i ) \cong \bigoplus _{i=1}^k \QQ .$$ 
Let $e_1,\ldots , e_k \in {\mathcal A}$ denote the primitive idempotents. 
Then the unique maximal order ${\mathcal O}$ in ${\mathcal A}$ 
is $$ {\mathcal O  } = \bigoplus _{i=1}^k Re_i \cong \bigoplus _{i=1}^k \ZZ .$$
The
 two matrices $X$ and $Y$ in $\ZZ ^{n\times n}$ 
with minimal polynomial 
$\mu _X = \mu_Y = \prod _{i=1}^k (t-a_i) $ define two $R$-module 
structures $M_X$ and $M_Y$ on $\ZZ ^{1\times n}$ by letting $t$ act as right multiplication by
$X$ respectively $Y$. 

\begin{remark}\label{isoconj}
	$C(X) = \{ A\in \ZZ ^{n\times n } \mid AX = XA \} \cong \End_R(M_X)$
	and 
	$C(X,Y) \cong \Hom _R (M_Y,M_X) $. 
	In particular 
any isomorphism between the two $R$-modules  $M_X$ and $M_Y$ 
is given by a matrix $T\in \GL_n(\ZZ )$  conjugating $X$ to $Y$. 
\end{remark}

Applying Remark \ref{isoconj} to the localizations of $M_X$ and $M_Y$, the 
 matrices $T_p \in \GL_n(\ZZ _{(p)}) $ conjugating $X$ to $Y$ yield isomorphisms 
between these localizations  for all primes $p$. 
So $M_X$ and $M_Y$ are in the same genus of $R$-lattices. 

The ${\mathcal O}$-module 
$$\Gamma := M_X {\mathcal O}   = \bigoplus _{i=1}^k M_X e_i =: \bigoplus _{i=1}^k \Gamma _i $$
 has endomorphism ring 
$$\Delta := \End_{{\mathcal O}} (\Gamma ) \cong \bigoplus _{i=1}^k 
\ZZ ^{n_i\times n_i } $$ where $n_i = \dim (\Gamma_i ) $.
In particular the genus of the $\Delta $-lattice $\Gamma $ consists of 
a single class, and hence by \cite[Satz 3]{Jacobinski} the genus
of the $R$-lattice $M_X$ consists of a single narrow genus.  

Put $\Lambda _i := M_X\cap \Gamma _i$. Then
$$ \Gamma = \bigoplus _{i=1}^k \Gamma _i \supseteq M_X \supseteq \bigoplus _{i=1}^k \Lambda _i $$
	and $X$ acts on $\Gamma _i$ and on $\Lambda _i$ as a scalar matrix, the multiplication by $a_i$. 
	Recall that we choose $q_i \in \NN $ to be minimal such that $E_i = q_i e_i \in \End_{\ZZ } (M_X) = \ZZ ^{n\times n}$.  
	\begin{remark} \label{elt}
		If $(d_1,\ldots , d_{n_i})$ are the abelian invariants of $E_i$  and $m_j:=q_i/d_j $ for $i=1,\ldots ,n_i$, then
		$$\Gamma _i/\Lambda _i \cong \ZZ/m_{1} \ZZ \times \ldots \times  \ZZ /m_{n_i} \ZZ .$$
	\end{remark}

	To agree with the notation in \cite{Jacobinski} we put 
	$H:=C(X) = \End _{R } (M_X)\subseteq \Delta $. Then 
	$\Delta $ is a maximal order containing $H$ and 
	 the maximal two-sided $\Delta $-ideal contained in $H$ is 
	$$ {\mathcal F} := \bigoplus _{i=1}^k E_i \Delta  = \bigoplus _{i=1}^k q_i \ZZ ^{n_i\times n_i} \subseteq C(X) \subseteq \Delta.$$
	Moreover $M_X {\mathcal F} = \Gamma {\mathcal F} = \bigoplus _{i=1}^k 
	q_i \Gamma _i $. 
In the notation preceding \cite[Satz 4]{Jacobinski} we put 
$$%\begin{array}{l} 
%	H:=\End_{R} (M_X) = C(X) , \\ \Delta := \End_{\mathcal O} (_\Gamma) = \bigoplus _{i=1}^k 
%	\ZZ ^{n_i\times n_i }
% \\
\tilde{\Delta } : = \Delta / {\mathcal F} \mbox{ and }  \tilde{H} := C(X) / {\mathcal F} . % \end{array} 
$$
Then $\tilde{H} \leq \tilde{\Delta } $.
The respective groups of units are 
$$ \begin{array}{l}  U(\Delta ) = \prod _{i=1}^k \GL _{n_i} (\ZZ ) = \prod_{i=1}^k \GL (\Gamma _i), \\
	U(\tilde{\Delta }) = \prod _{i=1}^k \GL _{n_i} (\ZZ / q_i \ZZ ) = \GL(\Gamma / {\mathcal F}\Gamma ) ,\mbox{ and } \\ 
	U(\tilde{H}) = U(C(X)/{\mathcal F}) = \{ g\in U(\tilde{\Delta }) \mid (M_X/M_X {\mathcal F} )  g = M_X / M_X {\mathcal F}  \} .  
\end{array} $$
We also put $\widetilde{U(\Delta )} := U(\Delta )/{\mathcal F} \leq U(\tilde{\Delta })$ to denote the 
reduction of the units of $\Delta $ modulo ${\mathcal F}$.

Then \cite[Satz 4]{Jacobinski} tells us that 
the isomorphism classes of $C(X)$-lattices in the (narrow)
genus of $M_X$ correspond bijectively to the double cosets 
$$U(\tilde{H}) \slash U(\tilde{\Delta } ) \backslash \widetilde{U(\Delta)} .$$
So to prove Theorem \ref{main} we need to show that this set consists of only one element.

\begin{lemma} 
	In the situation of Theorem \ref{main} we have $|U(\tilde{H}) \slash U(\tilde{\Delta } ) \backslash \widetilde{U(\Delta)}| = 1 .$
\end{lemma} 

\begin{proof}
	Clearly $C(X) = H \leq \Delta $, so we may write any element $B$ of $H$ as a tuple 
	$(B_1,\ldots ,B_k)$ of matrices $B_i \in \ZZ ^{n_i \times n_i} = \End _{\ZZ }(\Gamma _i ) $ which will be our 
	canonical notation for the elements of $\Delta = \oplus_{i=1}^k \End _{\ZZ }(\Gamma _i)$. 
	In particular 
	$$ H = \{ B:=(B_1,\ldots ,B_k) \in \Delta \mid  M_X B  \subseteq M_X \} .$$
	Let $\tilde{A}:=(\tilde{A}_1,\ldots , \tilde{A}_k) \in U(\tilde{\Delta }) $ and choose a preimage $A=(A_1,\ldots , A_k) \in \Delta $, so
	$A_i \in \ZZ ^{n_i\times n_i} $ reducing modulo $q_i$ to $\tilde{A}_i$. 
	Then $d_i:=\det(A_i ) \in \ZZ $ maps onto a unit $\det (\tilde{A}_i) \in \ZZ/q_i\ZZ $. Let $d_i'\in \ZZ $ with
	$d_i d_i' \equiv 1 \pmod{q_i} $ be the corresponding inverse.

	We construct  $B =(B_1,\ldots , B_k) \in H$  such that 
$\det(B _i ) \equiv d'_i  \pmod{q_i} $ for all $i$.

	Assume that part (a) of Assumption \ref{ass} holds. 
	If $m_1,\ldots , m_{n_i}$ are as in Remark \ref{elt} 
	there is a basis $(b^{(i)}_1,\ldots , b^{(i)}_{n_i})$ of $\Gamma _i$ such that 
	$$( m_1b^{(i)}_1,\ldots ,  m_{n_i} b^{(i)}_{n_i}) $$ is a basis of $\Lambda  _i$. 
	By Assumption \ref{ass} we have $m_{n_i} = 1$ for all $i$.
	Put $$K_i:= \langle b^{(i)}_{n_i} \rangle \mbox{  and } K_i':=  \langle b^{(i)}_1,\ldots , b^{(i)}_{n_i-1} \rangle  . $$
	Then $\Gamma _i = K_i \oplus K_i'$ and $\Lambda _i = K_i \oplus (K_i'\cap \Lambda )$. 
	Let $$K:= \bigoplus _{i=1}^k K _i \mbox{ and } K' :=  (\bigoplus _{i=1}^k K' _i ) \cap M_X .$$
	Then $M_X = K \oplus K'$ is a direct sum of these two $R$-sublattices. 

	Let $B$ be the endomorphism of $M_X$ that
	is the identity on $K'$ and the multiplication by $d'_i$ on $K_i$ for all
$i=1,\ldots ,k$.
Then $B = (B_1,\ldots , B_k) \in C(X) $  and 
$\det(B_i) =  d'_i $ for all $i$.
	\\
	If part (b) of Assumption \ref{ass} holds then we may 
	first add a multiple of $q_1$ to $d_1'$ such that $d_1'$ is 
	prime to $q_i$ for all $i=2,\ldots ,k$. 
	With the same construction as before we then find 
	$B'=(B'_1,\ldots , B'_k)  \in C(X)$ with $\det (B'_i) \equiv  d'_i/(d'_1)^{n_i} \pmod{q_i}$ for 
	all $i=2,\ldots, k$ and $\det(B'_1) = 1$.
	Then $B:=d'_1 B'\in C(X)$  has the desired properties.

	In both cases 
 $\tilde{B} \in U(\tilde{H}) $ and
	$BA=(B_1A_1,\ldots , B_kA_k) \in \Delta $ satisfies $\det (B_iA_i) \equiv 1 \pmod{q_i} $, so $\widetilde{B_iA_i} \in \SL_{n_i}(\ZZ/q_i\ZZ ) $. 
	By Lemma \ref{b}, there are matrices $C_i \in \SL_{n_i}(\ZZ )$ with 
	$\widetilde{C_i^{-1}} = \widetilde{B_iA_i}$ for all $i$. 
	Then $C:=(C_1,\ldots ,C_k) \in U(\Delta ) $ satisfies $\tilde{B} \tilde{A} \tilde{C} = 1 $. 
\end{proof}

%\begin{remark} \label{det}
	%The proof above has shown that in the situation of Theorem \ref{main} for any 
	%$q\in \NN $ and any $(d_1,\ldots , d_k )\in \NN ^k$ such that $d_i$ are prime to $q$ 
	%there is $B=(B_1,\ldots, B_k) \in C(X) $ such that $\det(B_i) \equiv d_i \pmod{q}$ for all $i$.
%\end{remark}

\section{Paley and Peisert} \label{Paley}

This last section is dedicated to the proof that the adjacency matrices of the 
Paley and Peisert graphs satisfy Part (b) of Assumption \ref{ass}.

Let $p$ be a prime $p\equiv 3 \pmod{4}$ and $q:=p^{2t}$ be an even power of $p$.  
The Paley graph (see \cite[p. 101]{Brouwer}) and the Peisert graph \cite{Peisert} on $q$ vertices are two 
cospectral Cayley graphs on an elementary abelian group of order $q$ which are isomorphic if and only if $q=9$ (see \cite{Sin}).
Choose a primitive element $\beta \in \FF _q^{\times }$ and let $U:= \langle \beta ^4 \rangle \leq \FF _q^{\times } $ 
denote the subgroup of fourth powers in the multiplicative group $\FF _q^{\times } $ of the field with $q$ elements. 
Then $$\FF_q^{\times } = U \cup \beta U \cup \beta ^2 U \cup \beta ^3 U .$$
The Paley graph $P(q)$ and the Peisert graph $P^*(q)$ have vertex set $\FF _q$.
Two vertices $i,j\in \FF _q$ are joined in $P(q)$, if and only if $i-j\in U\cup \beta ^2 U =: S = (\FF _q^{\times } )^2 $ 
and in $P^*(q)$ is and only if $i-j \in U\cup \beta U $.
Let $A(q) $ respectively $A^*(q)$ denote the adjacency matrices of $P(q)$ respectively $P^*(q)$. 

One  main result of \cite{Sin} is that for $q=p^2$ the adjacency matrices 
$A(q)$ and $A^*(q)$ are conjugate in $\GL_{q}(\ZZ _{(\ell )} )$ for all primes $\ell $. 

Using Theorem \ref{main} this allows us to show the following result:

\begin{theorem} \label{conjugate}
	The matrices $A(p^2)$ and $A^*(p^2)$ are conjugate in $\GL_{p^2}(\ZZ )$.
\end{theorem}

To prove the theorem 
we show that  the matrix $X:=A(p^2)$ satisfies part (b) of Assumption \ref{ass}. 
Put 
$$k:=\frac{p^2-1}{2}, r:= \frac{p-1}{2}, s:= \frac{-p-1}{2} .$$ 
Then the eigenvalues of $X$ are $k,r,s$ with multiplicities $1,\frac{p^2-1}{2}, \frac{p^2-1}{2}$. 
Define
$$\begin{array}{lll}
	E_1 := &  2(X-rI)(X-sI) /k = & J \\  E_2 := &-(X-kI)(X-sI) /r = &  sJ+pX-psI \\ E_3 := & -(X-kI)(X-rI)/s =& rJ-pX+prI
\end{array} $$
where $I$ denotes the unit matrix and  $J$ the all-ones matrix.
Then elementary computations show that for $i\neq j \in \{1,2,3\}$ 
$$E_i^2 = p^2 E_i \mbox{ and } E_iE_j=0. $$  In particular $e_i:=\frac{1}{p^2}E_i$ are the primitive
idempotents in $\QQ [X]$ and $q_i = p^2$ for $i=1,2,3$. Moreover 
		 $\rk (E_1) =1 $ and hence $\rk (E_2) = \rk (E_3 ) = k $.
The next lemma shows that the $E_i$ satisfy part (b) of Assumption \ref{ass}. 
Therefore Theorem \ref{main} together with the local considerations in \cite{Sin} imply Theorem \ref{conjugate}.

\begin{lemma} 
	For $i=2,3$ the Smith group of $E_i$ is 
	$$\ZZ/\ZZ \oplus (\ZZ/p\ZZ )^{(p+1)^2/4-2} \oplus (\ZZ/p^2\ZZ )^{(p-1)^2/4} .$$
\end{lemma}

\begin{proof}
	We use the methods from \cite{SinPaley}. 
	We first note that the exponent of the Smith group of $E_i$ divides $p^2$ by Remark \ref{elt}. 
	In particular we may pass to the $p$-adics. 
	Let $R:=\ZZ _p[\zeta _{q-1}]$ denote the ring of integers in the unramified extension of $\QQ _{p}$ of degree 2. 
	Then the adjacency matrix $X$ of $P(p^2)$ is seen as an endomorphism of $R[\FF _q] $. 
	Recall that 
	$S= \langle \beta ^2 \rangle = (\FF _q^{\times })^2 .$ 
	Then $S$ acts on $R[\FF _q] $ permuting the basis vectors
	$([x],s) \mapsto [xs] $ 
	for all $x\in \FF _q$, $s\in S$. 
	As $|S| = \frac{q-1}{2} \in R^{\times }$ is invertible in $R$ the
	$RS$-module  $R[\FF _q]$ is semisimple. 
        Let $\tau  : \FF _q^{\times } \to R^{\times } $ denote the group monomorphism known as the Teichm\"uller character.
	The matrices $J: [x] \mapsto \sum_{y\in \FF _q} [y] $ and 
	$X : [x] \mapsto \sum _{s\in S} [x+s] $ commute with the action of $S$ and hence act on the homogeneous 
	components 
	$$M_0 := \langle  {\bf 1} := \sum_{y\in \FF _q} [y], [0], b_k := \sum _{s\in S}[s] - \sum _{x\in \FF_q^{\times }\setminus S} [x] \rangle $$
	and 
	$$M_j := \langle b_j, b_{j+k} \rangle , \ j=1,\ldots , k-1 ,\mbox{ where }b_j := \sum _{x\in \FF _q^{\times }} \tau ^j(x^{-1}) [x] $$ (see \cite[Section 3]{SinPaley}).
	For the action of $E_i$ on  $M_0$ we compute 
	$$  
E_1^{(0)}  := 	 \left(\begin{array}{ccc} p^2 & 0 & 0 \\ 1 & 0 & 0 \\ 0 & 0 & 0 \end{array}\right) ,\ 
	E_2 ^{(0) } := \frac{1}{2} \left(\begin{array}{ccc} 
	0 & 0 & 0 \\ \mbox{-}1 & p^2 & p \\ \mbox{-}p & p^3 & p^2 \end{array} \right) ,\  
			\ E_3^{(0)} = 
			 \frac{1}{2} \left(\begin{array}{ccc} 
			 0 & 0 & 0 \\ \mbox{-}1 & p^2 & \mbox{-}p \\ p & \mbox{-}p^3 & p^2 \end{array} \right) .$$
In particular the rank of $E_i^{(0)}$ is 1, contributing a 1 to the abelian invariants
of $E_i$ for $i=1,2,3$. 
Clearly $J=E_1$ acts on $M_j$ as $0$ for $j\geq 1$.
	In the notation of \cite{SinPaley} let $j>0$ and $\alpha_j:= J(\tau ^{-j},\tau ^k)$ denote the Jacobi sum. 
	Then \cite[Lemma 3.1]{SinPaley} shows that $X$ acts on $M_j$ as right multiplication by 
$ X_j := \frac{1}{2} \left( \begin{array}{cc} -1 & \alpha _j \\ \alpha _{j+k} & -1 \end{array} \right) $ so $E_2$ and $E_3$ 
	by right multiplication with $p(X_j-s) $ respectively $-p(X_j-r)$ in matrices 
$$  E_2^{(j)} :=  \frac{p}{2} \left( \begin{array}{cc} p & \alpha _j \\ \alpha _{j+k} & p \end{array} \right) \mbox{ and } 
E_3^{(j)} := \frac{p}{2} \left( \begin{array}{cc} p & -\alpha _j \\ -\alpha _{j+k} & p \end{array} \right)  .$$
As the rank of $E_2$ and $E_3$ is $(p^2-1)/2 = 1+(k-1)$ and all $E_i^{(j)}$ are non-zero 
for $i=2,3$, $j=1,\ldots , k-1$ we obtain
that all these $E_i^{(j)}$ have rank 1, in particular $\alpha _j \alpha _{j+k} = p^2$,
for all $j=1,\ldots , k-1$. 
	Now \cite[Theorem 3.4]{SinPaley} says that 
	the $p$-adic valuation of $\alpha _i $ is 
	$$c(j) = \frac{1}{p-1} (s(j) + s(k) - s(j+k) ) $$ 
	where $s(j) = a+b$ if $j\equiv ap+b \pmod{p^2}$ with $0\leq a,b\leq p-1$.
	As $k=\frac{p-1}{2} p + \frac{p-1}{2} $ we have $s(k) = p-1$. 
	Moreover for 
	$$1\leq j = ap+b < \frac{p^2-1}{2} = k $$ we have 
	$a\leq (p-1)/2$ and $a\leq (p-3)/2$ if $b\geq (p-1)/2$. 
	Computing the digits of $j+k$ for these $j$ we find 
	\begin{itemize}
		\item[(0)]
			$s(j+k) = s(j) + s(k) $ if $0\leq a,b \leq \frac{p-1}{2}$, $(a,b) \not\in \{ (0,0) , (\frac{p-1}{2},\frac{p-1}{2} ) \} $
		\item[(1)] 
			$s(j+k) = s(j)+s(k) - (p-1) $ 
			if $\frac{p+1}{2} \leq b \leq p-1 $ and $0\leq a \leq \frac{p-3}{2} $.
	\end{itemize}
	So there are $(\frac{p+1}{2})^2-2$ such $1\leq j < k $ with 
	$c(j) = 0$ and $(\frac{p-1}{2} )^2$ such $j$ with $c(j) = 1$. 
	For the $j$ with $c(j) = 1 $ (and hence also $c(j+k) = 1$ as $\alpha_j\alpha_{j+k} = p^2$) 
	all entries of $E_i^{(j)}$ are divisible by $p^2$ so these  $j$ contribute 
	a value $p^2$ to the abelian invariants of both, $E_2$ and $E_3$. 
	If $c(j) = 0$ there is one entry of $E_i^{(j)}$ having valuation 1, so these $j$ contribute 
	a value $p$ to the abelian invariants of $E_2$ and $E_3$. 
\end{proof}

\end{document}